\documentclass{amsart}

\usepackage{amssymb}
\usepackage{graphicx}

\theoremstyle{plain}
\newtheorem{theorem}{Theorem}[section]

\newtheorem{corollary}[theorem]{Corollary}
\newtheorem{proposition}[theorem]{Proposition}

\theoremstyle{definition}
\newtheorem{definition}[theorem]{Definition}
\newtheorem{problem}[theorem]{Problem}
\newtheorem{example}[theorem]{Example}
\theoremstyle{remark}
\newtheorem*{remark}{Remark}

\newcommand{\cI}{\mathcal{I}} 
\newcommand{\I}{\cI} 
\newcommand{\cJ}{\mathcal{J}} 
\newcommand{\J}{\cJ} 
\newcommand{\eu}{\mathcal{EU}}
\DeclareMathOperator{\exh}{Exh}
\DeclareMathOperator{\fin}{Fin}
\DeclareMathOperator{\fix}{Fix}
\newcommand{\R}{\mathbb{R}} 
\newcommand{\cP}{\mathcal{P}} 

\renewcommand{\subset}{\subseteq}

\begin{document}

\title{Homogeneous ideals on countable sets}

\author{Adam Kwela}
\address{Institute of Mathematics, Faculty of Mathematics, Physics and Informatics, University of Gda\'{n}sk, ul.~Wita Stwosza 57, 80-308 Gda\'{n}sk, Poland}
\email{adam.kwela@ug.edu.pl}

\author{Jacek Tryba}
\address{Institute of Mathematics, Faculty of Mathematics, Physics and Informatics, University of Gda\'{n}sk, ul.~Wita Stwosza 57, 80-308 Gda\'{n}sk, Poland}
\email{jtryba@mat.ug.edu.pl}

\date{\today}


\thanks{The first author was supported by the grant BW-538-5100-B298-16.}


\begin{abstract}
We say that an ideal $\I$ on $\omega$ is homogeneous, if its restriction to any $\I$-positive subset of $\omega$ is isomorphic to $\I$. The paper investigates basic properties of this notion -- we give examples of homogeneous ideals and present some applications to topology and ideal convergence. Moreover, we answer questions related to our research posed in \cite{Inv}.
\end{abstract}

\maketitle

\section{Introduction}

Let $\omega$ stand for the set $\{0,1,2,\ldots\}$. A collection $\mathcal{I}\subset\mathcal{P}(X)$ is an \emph{ideal on $X$} if it is closed under finite unions and subsets. We additionally assume that $\mathcal{P}(X)$ is not an ideal and each ideal contains the family of all finite subsets of $X$. In this paper $X$ will always be a countable set. $\fin$ is the ideal of all finite subsets of $\omega$. The \emph{restriction of the ideal $\I$ to $X\subseteq\bigcup\I$} is given by $\I|X=\{A\cap X:\ A\in\I\}$. Ideal is \emph{dense} if every infinite set contains an infinite subset belonging to the ideal. The \emph{filter dual to the ideal $\mathcal{I}$} is the collection $\mathcal{I}^*=\left\{A\subset X:A^c\in\mathcal{I}\right\}$ and $\mathcal{I}^+=\left\{A\subset X:A\notin\mathcal{I}\right\}$ is the collection of all \emph{$\mathcal{I}$-positive sets}. An ideal $\I$ on $\omega$ is \emph{maximal} if for any $A\subseteq\omega$ either $A\in\I$ or $\omega\setminus A\in\I$ (equivalently: $\I$ is maximal with respect to $\subseteq$). 

Ideals $\mathcal{I}$ and $\mathcal{J}$ are \emph{isomorphic} ($\I\cong\J$) if there is a bijection $f\colon\bigcup\mathcal{J}\to\bigcup\mathcal{I}$ such that 
$$A\in\mathcal{I} \Longleftrightarrow f^{-1}[A]\in\mathcal{J}.$$

In this paper, for a given ideal $\I$, we investigate family of sets $H(\I)$ such that the restrictions of $\I$ to members of this family are isomorphic to $\I$.

\begin{definition}
Let $\I$ be an ideal on $\omega$. Then
$$H(\I)=\left\{A\subseteq\omega:\ \mathcal{I}|A\cong \mathcal{I}\right\}$$
is called \emph{the homogeneity family of the ideal $\mathcal{I}$}. 
\end{definition}

We call an ideal \emph{admissible} if it is not isomorphic to $$\fin\oplus\mathcal{P}(\omega)=\{A\subseteq\{0,1\}\times\omega:\ \{n\in\omega:\ (1,n)\in A\}\in\mathbf{Fin}\}.$$

\begin{proposition}\
\label{zawierania}
\begin{itemize}
	\item[(a)]$H(\I)\subseteq\I^+$ for any ideal $\I$.
	\item[(b)]$H(\I)\supseteq\I^\star$ for any admissible ideal $\I$.
\end{itemize}
\end{proposition}

\begin{proof}
To prove part (a) consider any $A\in\I$. Then $\I|A=\mathcal{P}(A)$. Therefore, $\I|A\not \cong\I$.

To prove part (b) consider any $A\in\I^\star$. If $\I\cong \fin$, then obviously $\I|A\cong\I$. So suppose now that $\I\not \cong \fin$. Then there is an infinite $B\subseteq A$ with $B\in\I$ (since $\I$ is not isomorphic to $\fin\oplus\mathcal{P}(\omega)$). Let $f\colon\omega\to A$ be such that $f(x)=x$ for any $x\in A\setminus B$ and $f\upharpoonright (B\cup A^c)$ is any bijection between $B\cup A^c$ and $B$. Then $f$ witnesses that $\mathcal{I}|A\cong \mathcal{I}$.
\end{proof}

The above proposition indicates that we can introduce two classes of ideals. Members of those classes have critical homogeneity families.

\begin{definition}
We call an ideal $\mathcal{I}$ on $\omega$:
\begin{itemize}
	\item \emph{homogeneous}, if $H(\I)=\I^+$;
	\item \emph{anti-homogeneous}, if $H(\I)=\I^\star$.
\end{itemize}
\end{definition}

\begin{remark}
In \cite[Section 5]{Fremlin} the notion of homogeneous filters is introduced. Note that an ideal is homogeneous if and only if its dual filter is homogeneous. 
\end{remark}

\begin{example}
The only ideals that are both homogeneous and anti-homogeneous are maximal ideals.
\end{example}

The space $2^\omega$ of all functions $f:\omega\rightarrow 2$ is equipped with the product topology (each space $2=\left\{0,1\right\}$ carries the discrete topology). We treat $\mathcal{P}(\omega)$ as the space $2^\omega$ by identifying subsets of $\omega$ with their characteristic functions. All topological and descriptive notion in the context of ideals on $\omega$ will refer to this topology.

In Section \ref{homogeneity} we give more examples of homogeneous ideals. Examples of anti-homogeneous ideals can be found in Section \ref{maximal}, however, all of them are based on maximal ideals and hence are not Borel or even analytic. We give an example of an $\mathcal{F}_{\sigma\delta}$ anti-homogeneous ideal in Section \ref{submeasures}.

In this paper we answer some questions posed in \cite{Inv}, where $\I$-invariant and bi-$\I$-invariant injections were investigated.

\begin{definition}
Let $\I$ be an ideal on $\omega$ and $f\colon\omega\to\omega$ be an injection. We say that $f$ is:
\begin{itemize}
	\item \emph{$\I$-invariant} if $f[A]\in\I$ for all $A\in \I$;
	\item \emph{bi-$\I$-invariant} if $f[A]\in\I \Longleftrightarrow A\in \I$ for all $A\subseteq\omega$.
\end{itemize}
\end{definition}

Next example shows that invariance and bi-invariance of an injection does not have to coincide.

\begin{example}
The \emph{ideal of sets of asymptotic density zero} is given by 
$$\mathcal{I}_d=\{A\subset\omega:\ \lim_{n\rightarrow\infty}\frac{|A\cap n|}{n}=0\}.$$ 
Note that every increasing injection is $\I_d$-invariant. In particular, $f\colon\omega\to\omega$ given by $f(n)=n^2$ is invariant. However, it is not bi-$\I_d$-invariant, since $f[\omega]\in\I_d$.
\end{example}

We say that \emph{$\mathcal{J}$ contains an isomorphic copy of $\mathcal{I}$} and write $\mathcal{I}\sqsubseteq\mathcal{J}$ if there is a bijection $f\colon\bigcup\mathcal{J}\to\bigcup\mathcal{I}$ such that
$$A\in\mathcal{I}\Longrightarrow f^{-1}[A]\in\mathcal{J}.$$
Note that $\I\cong\J$ implies $\mathcal{I}\sqsubseteq\mathcal{J}$, however, both $\mathcal{I}\sqsubseteq\mathcal{J}$ and $\mathcal{J}\sqsubseteq\mathcal{I}$ do not imply that $\I\cong\J$ (cf. \cite[Section 1]{K}).

The connection between invariant injections and homogeneity families is the following.

\begin{remark}\
\begin{itemize}
	\item If $f\colon\omega\to\omega$ is bi-$\I$-invariant then $f[\omega]\in H(\I)$. On the other hand, if $A\in H(\I)$ then there is a bi-$\I$-invariant $f\colon\omega\to\omega$ with $f[\omega]=A$.
	\item If $f\colon\omega\to\omega$ is $\I$-invariant then $\I\sqsubseteq\I|f[\omega]$. On the other hand, if $\I\sqsubseteq\I|A$ then there is an $\I$-invariant $f\colon\omega\to\omega$ with $f[\omega]=A$.
\end{itemize}
\end{remark}

The paper is organized as follows. In Section \ref{homogeneity} we investigate homogeneity families and homogeneous ideals. Section \ref{maximal} is devoted to answering \cite[Question 1]{Inv}. In Section \ref{submeasures} we investigate ideals induced by submeasures (summable ideals and Erd{\"o}s-Ulam ideals). In particular, we answer \cite[Question 2]{Inv}. The last part of our paper concerns applications of our results to ideal convergence. We answer \cite[Questions 3 and 4]{Inv} related to this topic.

\section{Homogeneous ideals}
\label{homogeneity}

In this section we investigate basic properties of homogeneity families, give some examples of homogeneous ideals and show an application to topology.

The next theorem enables us to simplify computations of homogeneity families of ideals.

\begin{theorem}
\label{supersets}
The homogeneity family of any ideal is closed under supersets.
\end{theorem}

\begin{proof}
Take any $A\in H(\I)$ and $A\subseteq B$. There is a bijection $f\colon\omega\to A$ such that $X\in\I \Leftrightarrow f[X]\in\I$. Denote
$$A'=\left\{a\in A:\ \exists_{n\in\omega}\left(\forall_{k<n}\ f^{-k}(a)\in A\right)\ \wedge\ f^{-n}(a)\in B\setminus A\right\}$$
and $M=A'\cup (B\setminus A)$. Define $\varphi\colon\omega\to B$ by:
$$\varphi(x)=\left\{\begin{array}{ll}
f(x), & \textrm{if } x\in \omega\setminus M,\\
x, & \textrm{if } x\in M.
\end{array}\right.$$
We will show that $\varphi$ is a bijection witnessing that $\I|B\cong\I$.

Firstly we show that $\varphi$ is $1-1$. Take $x,y\in\omega$, $x\neq y$. If $x,y\in M$ then $\varphi(x)=x\neq y=\varphi(y)$. If $x,y\notin M$, then $\varphi(x)=f(x)\neq f(y)=\varphi(y)$ since $f$ is $1-1$. So it remains to consider the case that $x\in M$ and $y\notin M$. Then $\varphi(x)=x$ and $\varphi(y)=f(y)$. Suppose that $f(y)=x$. Since $x\in M$, we have $x\in A'$ (the case $x\in B\setminus A$ is impossible by $f(y)\in f[\omega]=A$). Observe that $f^{-1}(x)=y\notin B\setminus A$ (since $y\notin M$). Therefore, $y\in A$. But then $y\in A'\subseteq M$. A contradiction.

Now we prove that $\varphi$ is onto. Let $y\in B$. If $y\in M$, then $\varphi(y)=y$ and we are done. So suppose that $y\notin M$ and observe that $f^{-1}(y)\notin M$. Indeed, otherwise either $f^{-1}(y)\in B\setminus A$ or $f^{-1}(y)\in A'$. In both cases we get that $y\in A'$ which contradicts $y\notin M$. Therefore, $f^{-1}(y)\notin M$ and we have $y=f(f^{-1}(y))=\varphi(f^{-1}(y))$.

Finally, we show that $\varphi$ witnesses $\I|B\cong\I$. Take any $X\in\I$. We have
$$\varphi[X]=\varphi[X\cap M]\cup\varphi[X\setminus M]=(X\cap M)\cup f[X\setminus M]\in \I,$$
$$\varphi^{-1}[X]=\varphi^{-1}[X\cap M]\cup\varphi^{-1}[X\setminus M]=(X\cap M)\cup f^{-1}[X\setminus M]\in \I$$
(since $f$ witnesses $\I|A\cong\I$). This finishes the proof.
\end{proof}

\begin{corollary}
\label{homogeneous}
The following are equivalent for any ideal $\I$ on $\omega$:
\begin{itemize}
	\item[(a)] $\I$ is homogeneous;
	\item[(b)] for each $B\notin\I$ there is $A\subseteq B$ such that $A\in H(\I)$.
\end{itemize}
\end{corollary}

\begin{proof}
\textbf{(a) $\Rightarrow$ (b): }Obvious.

\textbf{(b) $\Rightarrow$ (a): }Take any $B\notin\I$. By condition (b) there is $A\subseteq B$ such that $A\in H(\I)$. Then, by Theorem \ref{supersets} we have $B\in H(\I)$.
\end{proof}

Now we will give some examples of homogeneous ideals. 

\begin{example}
\label{Fin}
The ideal $\fin$ is the simplest example of a homogeneous ideal.
\end{example}

\begin{example}
Define $D=\{(i,j)\in\omega^2:\ i\geq j\}$ and 
$$\mathcal{ED}_{fin}=\left\{A\subseteq D:\ \exists_{m\in\omega}\forall_{k\in\omega} |A\cap\{k\}\times\omega|\leq m\right\}$$
(cf. \cite{Meza}). Then $\mathcal{ED}_{fin}$ is a homogeneous ideal. Indeed, set any $A\notin\mathcal{ED}_{fin}$ and for each $n\in\omega$ pick $k_n\in\omega$ and $a^n_1,\ldots,a^n_{n}\in A\cap\{k_n\}\times\omega$. We can additionally assume that the sequence $(k_n)_{n\in\omega}$ is increasing. Denote $A'=\{a^n_m:\ n\in\omega,m\leq n\}$. Then the bijection $f\colon D\to A'$ given by $f(n,m)=a^n_m$ witnesses that $\mathcal{ED}_{fin}\cong\mathcal{ED}_{fin}|A'$ and by Corollary \ref{homogeneous} we get that $\mathcal{ED}_{fin}$ is homogeneous.
\end{example}

\begin{example}
Given $n\in\omega$ by $[\omega]^n$ we denote the collection of all $n$-element subsets of $\omega$. For all $n\in\omega\setminus\{0\}$ define
$$\mathcal{R}_n=\left\{A\subseteq [\omega]^n:\ \forall_{B\subseteq\omega} \textrm{ if }B\textrm{ is infinite then }A\textrm{ does not contain }[B]^n\right\}$$
(note that $\mathcal{R}_1=\fin$). By Ramsey's Theorem each $\mathcal{R}_n$ is an ideal. Using Corollary \ref{homogeneous} it is easy to see that all $\mathcal{R}_n$ are homogeneous.
\end{example}

\begin{example}
Define the van der Waerden ideal
$$\mathcal{W}=\left\{A\subseteq\omega:\ \exists_{n\in\omega} A\textrm{ does not contain an arithmetic progression of length }n\right\}$$
and the Hindman ideal
$$\mathcal{H}=\left\{A\subseteq\omega:\ \forall_{B\subseteq\omega} \textrm{ if }B\textrm{ is infinite then }FS(B)\not\subseteq A\right\},$$
where $FS(B)=\left\{\sum_{n\in F}n:\ F\subseteq B\ \wedge\ F\in\fin\right\}$ (note that $\mathcal{W}$ is closed under finite unions by the van der Waerden's Theorem and $\mathcal{H}$ is closed under finite unions by the Hindman's Theorem). Both $\mathcal{W}$ and $\mathcal{H}$ are homogeneous. Indeed, by the proofs of \cite[Proposition 4]{vdW} and \cite[Theorems 3.3 and 4.5]{FT} both ideals satisfy condition (b) of Corollary \ref{homogeneous}.
\end{example}

\begin{example}
For every $n\in\omega\setminus\{0\}$ define the Gallai ideal
$$\mathcal{G}_n=\left\{A\subseteq\omega^n:\ \exists_{k\in\omega}\forall_{v\in\omega^n}\forall_{\alpha\in\omega\setminus\{0\}}\ v+\alpha\cdot\{1,2,\ldots,k\}^n\not\subseteq A\right\}$$
(note that each $\mathcal{G}_n$ is closed under finite unions by the Gallai's Theorem). In particular, $\mathcal{G}_1=\mathcal{W}$. It can be shown that each $\mathcal{G}_n$ is homogeneous. We provide only a sketch of the proof for $n=2$.

Let $A\notin\mathcal{G}_2$ and fix any bijection $h\colon\omega^2\to\omega$. Let $\textrm{pr}_1$ denote the projection on the first coordinate. Construct inductively $v^i_j\in\omega^2$ and $\alpha^i_j\in\omega\setminus\{0\}$ for all $i,j\in\omega$ such that:
\begin{itemize}
	\item[(i)] $A^i_j=v^i_j+\alpha^i_j\cdot\{1,2,\ldots,2^{i}\}^2\subseteq A$;
	\item[(ii)] $2\cdot\max (\textrm{pr}_1(A^i_j))<\min (\textrm{pr}_1(A^{i'}_{j'}))$ whenever $h(i,j)<h(i',j')$.
\end{itemize}
The construction is possible since given some $l\in\omega$ we have $\{l\}\times\omega\in\mathcal{G}_2$ and hence $A\setminus\left(\{0,1,\ldots,l\}\times\omega\right)\notin\mathcal{G}_2$. 

Let $A'=\bigcup_{i,j\in\omega}A^i_j\subseteq A$. By Corollary \ref{homogeneous} it suffices to show that $A'\in H(\mathcal{G}_2)$. 

We are ready to define the required isomorphism $f\colon A'\to\omega^2$. Given $(a,b)\in A^i_j$ find $k,l$ such that $(a,b)=v^i_j+\alpha^i_j\cdot (k,l)$. Then define $f(a,b)=(k+\sum_{n<i}2^n,l+2^i\cdot j)$.

\begin{figure}[h]
\begin{center}
\includegraphics[scale=0.30]{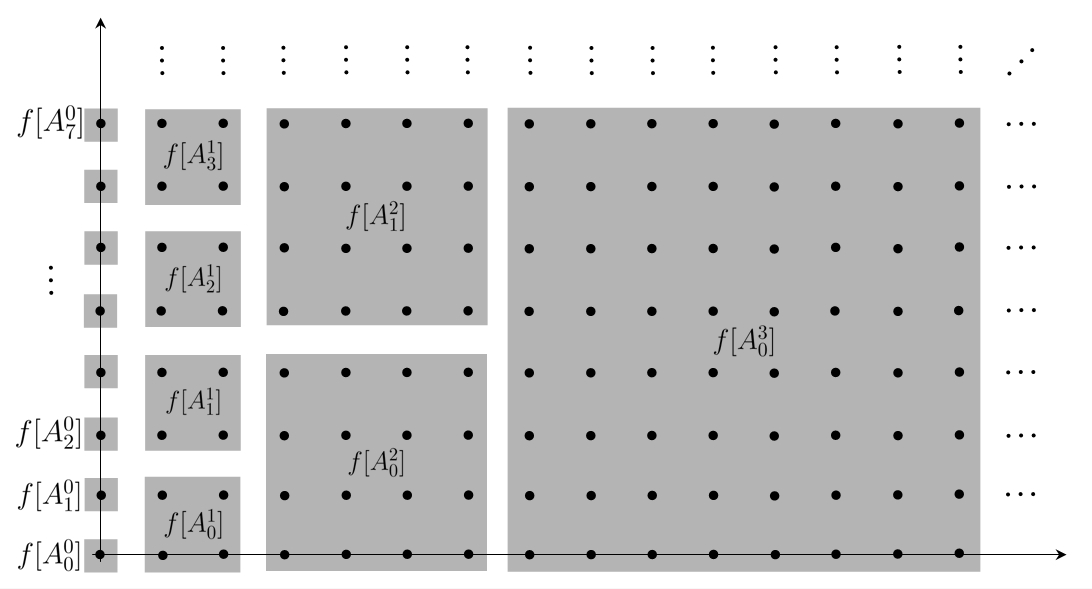}
\end{center}
\caption{Illustration of the image of function $f$.}
\end{figure}

To prove that $f$ witnesses $A'\in H(\mathcal{G}_2)$, set any $B\notin\mathcal{G}_2$. Assume first that $B\subseteq A'$. It can be shown that by condition (ii) actually for each $k$ there are $v$, $\alpha$ and $i$, $j$ such that $v+\alpha\cdot\{1,2,\ldots,k\}^2\subseteq A^i_j\cap B$. Since in this case $f[v+\alpha\cdot\{1,2,\ldots,k\}^2]$ is also of the form $v'+\alpha'\cdot\{1,2,\ldots,k\}^2$ for some $v'$ and $\alpha'$, we conclude that $f[B]\notin \mathcal{G}_2$.

To finish the proof we need to check that $f^{-1}[B]$ is not in $\mathcal{G}_2$ for every $B\notin\mathcal{G}_2$. Actually, it is enough to show that for every $v$ and $\alpha$ the set $f^{-1}[v+\alpha\cdot\{1,2,\ldots,8k\}^2]$ contains $v'+\alpha'\cdot\{1,2,\ldots,k\}^2$ for some $v'$ and $\alpha'$. Indeed, consider the set $\textrm{pr}_1[v+\alpha\cdot\{1,2,\ldots,8k\}^2]$. It is an arithmetic progression of length $8k$ and common difference $\alpha$. It can be shown that in this case there is $i$ such that $I_i\cap \textrm{pr}_1[v+\alpha\cdot\{1,2,\ldots,8k\}^2]$, where $I_i=[\sum_{n<i}2^n,(\sum_{n<i}2^n)+2^i-1]$, contains an arithmetic progression of length $2k$ and common difference $\alpha$ (cf. the proofs of \cite[Proposition 4]{vdW} and \cite[Theorems 3.3]{FT}). Therefore, there is $w$ such that $w+\alpha\cdot\{1,2,\ldots,2k\}^2\subseteq v+\alpha\cdot\{1,2,\ldots,8k\}^2$ and $\textrm{pr}_1[w+\alpha\cdot\{1,2,\ldots,2k\}^2]\subseteq I_i$. Now it suffices to observe that there are also $w'$ and $j$ with $w'+\alpha\cdot\{1,2,\ldots,k\}^2\subseteq I_i\times [j\cdot 2^i,(j+1)\cdot 2^i]$. It follows that $f^{-1}[v+\alpha\cdot\{1,2,\ldots,8k\}^2]$ contains $v'+\alpha'\cdot\{1,2,\ldots,k\}^2$ for some $v'$ and $\alpha'$.
\end{example}

The above examples suggest that the following should be true.

\begin{problem}
Define the Folkman ideal
$$\mathcal{F}=\left\{A\subseteq\omega:\ \exists_{n>1}\forall_{B\subseteq\omega} \textrm{ if }|B|=n \textrm{ then } FS(B)\not\subseteq A\right\}$$
(note that $\mathcal{F}$ is closed under finite unions by the Folkman's Theorem). Is $\mathcal{F}$ homogeneous? 
\end{problem}

Let $\I$ and $\J$ be two ideals on $X$ and $Y$, respectively. By $\mathcal{I}\otimes\mathcal{J}$ we denote the \emph{product of the ideals $\I$ and $\J$} given by:
$$A\in\mathcal{I}\otimes\mathcal{J}\ \Leftrightarrow\ \{x\in X:\ A_x\notin\mathcal{J}\}\in\mathcal{I},$$
for every $A\subseteq X\times Y$, where $A_x=\{y\in Y:\ (x,y)\in A\}$.

\begin{remark}
Observe that $\mathcal{G}_2\neq\mathcal{W}\otimes\mathcal{W}$, since the set $D=\{(i,j)\in\omega^2:\ i\geq j\}$ belongs to $\mathcal{W}\otimes\mathcal{W}$ and does not belong to $\mathcal{G}_2$ (in fact, one can also give an example of a set belonging to $\mathcal{G}_2\setminus\mathcal{W}\otimes\mathcal{W}$).
\end{remark}

The next result gives us many more examples of homogeneous ideals.

\begin{proposition}
\label{otimes}
If $\I$ and $\J$ are homogeneous ideals then so is $\I\otimes\J$.
\end{proposition}

\begin{proof}
Take any $A\notin\I\otimes\J$. Then the set $B=\{n\in\omega:\ A_n\notin\J\}$ does not belong to $\I$. Define
$$A'=\bigcup_{n\in B}\{n\}\times A_n\notin\I\otimes\J.$$
By Corollary \ref{homogeneous} it suffices to show that $A'\in H(\I\otimes\J)$. 

For each $n\in B$ there is a bijection $f_n\colon A_n\to\omega$ such that $X\in J\Leftrightarrow f_n^{-1}[X]\in \J|A_n$ for each $X\subseteq\omega$. There is also a bijection $g\colon B\to\omega$ such that $X\in I\Leftrightarrow g^{-1}[X]\in \I|B$ for each $X\subseteq\omega$. Define $\varphi\colon A'\to\omega\times\omega$ by $\varphi((i,j))=(g(i),f_i(j))$ for each $(i,j)\in A'$. Then $\varphi$ is a bijection witnessing that $\I\otimes\J\cong \I\otimes\J|A'$.
\end{proof}

\begin{remark}
Let $\fin^1=\fin$ and $\fin^{n+1}=\fin\otimes\fin^n$ for all $n>1$. By Example \ref{Fin} and Proposition \ref{otimes} all ideals $\fin^n$, for $n\geq 1$, are homogeneous. Note that this fact is also shown in \cite[Section 5]{Fremlin}.
\end{remark}

Now we proceed to some applications of our results to topology. 

Let $\I$ be an ideal on $\omega$. A sequence of reals $(x_n)_{n\in \omega}$ is \emph{$\I$-convergent} to $x\in \R$, if $\{n\in \omega:\ |x_n-x|\geq \varepsilon\}\in \I$ for any $\varepsilon>0$. We say that a pair $(X,\I)$, where $X$ is a topological space and $\I$ is an ideal on $\omega$, has:
\begin{itemize}
	\item \emph{BW property}, if every sequence $(x_n)_{n\in\omega}\subseteq X$ has an $\I$-convergent subsequence $(x_n)_{n\in A}$ with $A\notin\I$;
	\item \emph{FinBW property}, if every sequence $(x_n)_{n\in\omega}\subseteq X$ has a convergent subsequence $(x_n)_{n\in A}$ with $A\notin\I$;
	\item \emph{hBW property}, if every sequence $(x_n)_{n\in A}\subseteq X$ with $A\notin\I$ has an $\I$-convergent subsequence $(x_n)_{n\in B}$ with $B\notin\I$ and $B\subseteq A$;
	\item \emph{hFinBW property}, if every sequence $(x_n)_{n\in A}\subseteq X$ with $A\notin\I$ has a convergent subsequence $(x_n)_{n\in B}$ with $B\notin\I$ and $B\subseteq A$.
\end{itemize}

\begin{proposition}
\label{top}
Suppose that $X$ and $Y$ are topological spaces and $\I$ is a homogeneous ideal. 
\begin{itemize}
	\item[(a)] $(X,\I)$ has the BW property if and only if $(X,\I)$ has the hBW property;
	\item[(b)] $(X,\I)$ has the FinBW property if and only if $(X,\I)$ has the hFinBW property;
	\item[(c)] if both $(X,\I)$ and $(Z,\I)$ have the BW property (FinBw property), then so does $(X\times Y,\I)$.
\end{itemize}
\end{proposition}

\begin{proof}
Straightforward.
\end{proof}

\begin{remark}
Fact \ref{top} generalizes results of \cite{FT}, \cite{Hindman} and \cite{vdW}, where (a), (b) and (c) were proved for $\I=\mathcal{W}$ and $\I=\mathcal{H}$.
\end{remark}

\section{Maximal ideals}
\label{maximal}

In this section we answer \cite[Question 1]{Inv}.

It is easy to see that for any ideal $\I$ and injection $f\colon\omega\to\omega$, if $\fix(f)\in\I^\star$ then $f$ is bi -$\I$-invariant. We say that an ideal $\I$ on $\omega$ satisfies condition $(C1)$, if the above implication can be reversed, i.e., for any bi-$\I$-invariant injection $f\colon\omega\to\omega$ we have $\fix(f)\in\I^\star$. Firstly we answer the first part of \cite[Question 1]{Inv} about characterization of the class of ideals satisfying condition $(C1)$.

\begin{theorem}
\label{bi-inv}
The following are equivalent for any ideal $\I$ on $\omega$:
\begin{itemize}
	\item[(a)] $\I$ does not satisfy condition $(C1)$;
	\item[(b)] there are $A,B\subseteq\omega$ such that $A\triangle B\notin\I$ and $\I|A\cong \I|B$.
\end{itemize}
\end{theorem}

\begin{proof}
\textbf{(a) $\Rightarrow$ (b): }By the assumption there is a bi-$\I$-invariant injection $f\colon\omega\to\omega$ such that $\fix(f)\notin\I^\star$. We can assume that $f[\omega]\notin\I$ (if $f[\omega]\in \I$, then $\omega=f^{-1}[f[\omega]]\in\I$, a contradiction). First we deal with the case $f[\omega]\notin\I^\star$. Define $A=\omega$ and $B=f[\omega]$. Then $A\triangle B\notin\I$ and $f$ witnesses that $\I|A=\I\cong \I|B$ since $f$ is bi-$\I$-invariant.

Assume now that $f[\omega]\in\I^\star$. We inductively pick points $a_n$ and $b_n$ for $n\in\omega$.  We start with $a_0=\min (\omega\setminus\fix(f))$ and $b_0=f(a_0)$. If all $a_k$ and $b_k$ for $k\leq n$ are defined, let 
$$a_{n+1}=\min \left(\omega\setminus (\fix(f)\cup \{a_k:\ k\leq n\}\cup \{b_k:\ k\leq n\}\cup f^{-1}[\{a_k:\ k\leq n\}])\right)$$
and $b_{n+1}=f(a_{n+1})$. Define $A=\{a_n:\ n\in\omega\}$ and $B=\{b_n:\ n\in\omega\}$. Then $A\cap B=\emptyset$. Moreover, $A\triangle B\notin\I$. Indeed, if $A\triangle B\in\I$ then also $f^{-1}[A]\in\I$, which contradicts $A\cup B\cup f^{-1}[A]=\omega\setminus\fix(f)\notin\I$. Finally, observe that $f\upharpoonright A$ is a bijection between $A$ and $B$ witnessing that $\I|A\cong \I|B$.

\textbf{(b) $\Rightarrow$ (a): }There is a bijection $f\colon A\to B$ witnessing $\I|A\cong \I|B$. Since $A\triangle B\notin\I$, either $A\setminus B\notin\I$ or $B\setminus A\notin\I$. Suppose that $A\setminus B\notin\I$ (the other case is similar). There are two possibilities.

If the set $A'=\{a\in A\setminus B:\ f(a)\in B\setminus A\}$ is not in $\I$, then define $g\colon\omega\to\omega$ by
$$g(x)=\left\{\begin{array}{ll}
f(x) & \textrm{if } x\in A'\\
f^{-1}(x) & \textrm{if } x\in f[A']\\
x & \textrm{if } x\notin A'\cup f[A']
\end{array}\right.$$
Then $g$ is a bi-$\I$-invariant injection and $\fix(g)^c\supseteq A'\notin\I$.

If the set $A''=\{a\in A\setminus B:\ f(a)\in A\cap B\}$ is not in $\I$, then define $g\colon\omega\to\omega$ by
$$g(x)=\left\{\begin{array}{ll}
f(x) & \textrm{if } x\in A''\\
f^{-1}(x) & \textrm{if } x\in f[A'']\\
x & \textrm{if } x\notin A''\cup f[A'']
\end{array}\right.$$
Then $g$ is a bi-$\I$-invariant injection and $\fix(g)^c\supseteq A''\notin\I$.
\end{proof}

If $\I$ and $\J$ are two ideals on $\omega$ then $\I\oplus\J$ is an ideal on $\{0,1\}\times\omega$ consisting of all $A\subseteq\{0,1\}\times\omega$ such that $\{n\in\omega:\ (0,n)\in A\}\in\I$ and $\{n\in\omega:\ (1,n)\in A\}\in\J$.

\begin{remark}
From the previous theorem it is easy to see that if some ideal satisfies condition $(C1)$, then it is anti-homogeneous. On the other hand, if $\I$ is any maximal ideal, then the ideal $\J=\I\oplus\I$ is anti-homogeneous, however, it does not satisfy condition $(C1)$. In Theorem \ref{ex} we construct an $\mathcal{F}_{\sigma\delta}$ anti-homogeneous ideal not satisfying condition $(C1)$.
\end{remark}

Now we proceed to examples of ideals satisfying condition $(C1)$ (which are also examples of anti-homogeneous ideals, by the above remark). The following two examples come from \cite{Inv}, however, we give a much simpler proofs.

\begin{example}[cf. {\cite[Corollary 9]{Inv}}]
\label{max}
Let $\I$ be any maximal ideal. Then $\I$ satisfies condition $(C1)$. 

Indeed, if $A\triangle B\notin\I$, then either $A\in\I$ and $B\in\I^\star$ or $B\in\I$ and $A\in\I^\star$. But then $\I|A$ and $\I|B$ cannot be isomorphic, since one of those ideals is isomorphic to $\I$, while the second one is isomorphic to $\mathcal{P}(\omega)$.
\end{example}

\begin{example}[cf. {\cite[Example 10]{Inv}}]
Let $\I$ and $\J$ be two non-isomorphic maximal ideals. Then the ideal $\I\oplus\J$ satisfies condition $(C1)$. 

Indeed, suppose that $A\triangle B\notin\I\oplus\J$ for some $A,B\subseteq\omega$. Without loss of generality we can assume that $A\cap (\{0\}\times\omega)\in\I$ but $A\cap (\{1\}\times\omega)\notin\J$ and $B\cap (\{0\}\times\omega)\notin\I$ but $B\cap (\{1\}\times\omega)\in\J$. Then we have $(\I\oplus\J)|A\cong\J$ and $(\I\oplus\J)|B\cong\I$. Hence, $\I|A\not\cong \I|B$. 
\end{example}

Next example is new and more complicated than the previous ones. Before presenting it, we need to introduce some notations.

If $(X_i)_{i\in I}$ is a family of sets, then $\sum_{i\in I}X_i$ denotes their disjoint sum, i.e., the set of all pairs $(i,x)$, where $i\in I$ and $x\in X_i$. For an ideal $\mathcal{J}$ on $\omega$ and a sequence $\left(\mathcal{J}_i\right)_{i\in \omega}$ of ideals on $\omega$ the family of all sets of the form $\sum_{i\in A}B_i\cup \sum_{i\in \omega\setminus A}\omega$, for $A\in\mathcal{J}^\star$ and $B_i\in\mathcal{J}_i$, constitutes a basis of an ideal on $\omega\times\omega$. We denote this ideal by $\mathcal{J}$-$\sum_{i\in \omega}\mathcal{J}_i$ and call \emph{$\mathcal{J}$-Fubini sum of the ideals $\left(\mathcal{J}_i\right)_{i\in I}$}.

\begin{example}
Let $(\J_i)_{i\in\omega}$ be a sequence of pairwise non-isomorphic maximal ideals on $\omega$. Then the ideal $\I=\fin$-$\sum_{i\in \omega}\mathcal{J}_i$ satisfies condition $(C1)$. 

Indeed, suppose that there are $A,B\subseteq\omega^2$ such that $A\triangle B\notin\I$ and $\I|A\cong \I|B$. Then $A,B\notin\I$ (if $A\in\I$, then $B\in\I$ by $\I|A\cong \I|B$, which contradicts $A\triangle B\notin\I$). Therefore, the set $R=\{n\in\omega:\ A_n\in\J^\star_n\}$ is infinite. Let $f\colon A\to B$ be the bijection witnessing that $\I|A\cong \I|B$. Define 
$$S=\{n\in R:\ \exists_{k(n)\in\omega} \left(f\left[\{n\}\times A_n\right]\right)_{k(n)}\in\J^\star_{k(n)}\}$$ 
and $T=R\setminus S$. Note that if $n\neq m$ and $n,m\in S$ then $k(n)\neq k(m)$. There are two possible cases.

{\bf Case 1.: }If $T\notin\fin$, then consider $f\left[\sum_{i\in T}\{i\}\times A_i\right]\notin\I|B$. Let $L$ consist of those $l\in\omega$ for which infinitely many $f\left[\{j\}\times A_j\right]$'s, for $j\in T$, intersect $\{l\}\times\omega$. Observe that $L\notin\fin$ and $L\times\omega\cap f\left[\sum_{i\in T}\{i\}\times A_i\right]\notin\I$. 

Firstly we will show that the set $T'$ consisting of those $j\in T$, for which there is some $l(j)\in L$ such that $(f^{-1}[\{l(j)\}\times\omega\cap B]\cap \{j\}\times A_j)_j\in J_j^\star$, is finite. Suppose otherwise and consider the case that there is some $l\in L$ such that $l=l(j)$ for infinitely many $j\in T'$. Then $X=\{l\}\times\omega\cap B\in\I|B$, but $f^{-1}[X]\notin\I|A$. A contradiction. On the other hand, if there is no such $l$, then the set 
$$X=\bigcup_{j\in T'}f^{-1}[\{l(j)\}\times\omega\cap B]\cap \{j\}\times A_j$$
is such that $X\notin\I|A$. However, $f[X]\in\I|B$ since 
$$(f[X])_{l(j)}=\bigcup_{ \{j'\in T':\ l(j')=l(j)  \} }(f[\{j'\}\times A_{j'}])_{l(j)}\in\J_{l(j)}$$ 
for all $j\in T'$ (by the fact that $j\in T$) and $(f[X])_i=\emptyset$ for $i\notin\{l(j):\ j\in T'\}$. Again we get a contradiction. Therefore, $T'$ is finite.

Let $T\setminus T'=\{t_0,t_1,\ldots\}$ and $L=\{l_0,l_1,\ldots\}$. Consider the set 
$$X=\bigcup_{i\in\omega}\left(f[\{t_i\}\times A_{t_i}]\setminus (\{0,1,\ldots,l_i\}\times\omega)\right).$$
Then $X\in\I|B$. To get a contradiction we need to show that $f^{-1}[X]\notin\I|A$. Indeed, it follows from the fact that for each $i\in\omega$ we have $$f^{-1}\left[X\right]\cap A_{t_i}=A_{t_i}\setminus\bigcup_{j<l_i}f^{-1}\left[\{l_j\}\times\omega\cap B\right]$$
and $f^{-1}\left[\{l_j\}\times\omega\cap B\right]\in\J_{l_i}$, since $t_i\notin T'$. 

{\bf Case 2.: }If $T\in\fin$, then $S\notin\fin$. Since $A\triangle B\notin\I$, we can pick $k(n)$'s in such a way that $S'=\{n\in S:\ n\neq k(n)\}\notin\fin$. Let $\{s_0,s_1,\ldots\}$ be an enumeration (without repetitions) of $S'$. Note that $\J_{s_i}|A_{s_i}\cong\J_{s_i}$ and $\J_{k(s_i)}|f[A_{s_i}]\cong \J_{k(s_i)}$ by maximality of all $\J_i$'s. Since the ideals $(\J_i)_{i\in\omega}$ are pairwise non-isomorphic, for each $i\in\omega$ one can find $X_{s_i}\subseteq A_{s_i}$ such that $X_{s_i}\in\J_{s_i}|A_{s_i}$ but $f[\{s_i\}\times X_{s_i}]\notin\J_{k(s_i)}|f[A_{s_i}]$ (if for some $i$ there is $Y_{s_i}\subseteq A_{s_i}$ such that $Y_{s_i}\notin\J_{s_i}|A_{s_i}$ but $f[\{s_i\}\times Y_{s_i}]\in\J_{k(s_i)}|f[A_{s_i}]$ then by maximality of $\J_i$'s it suffices to take $X_{s_i}=A_{s_i}\setminus Y_{s_i}$). Define $$X=\bigcup_{i\in\omega}\{s_i\}\times X_{s_i}.$$
Then obviously $X\in\I|A$. Moreover, $f[X]\notin\I|B$ (since $k(s_i)\neq k(s_j)$ for $i\neq j$). A contradiction. 
\end{example}

\begin{problem}
Is there a "nice" (for instance Borel or analytic) ideal satisfying condition $(C1)$?
\end{problem}

Now we proceed to answering the second part of \cite[Question 1]{Inv} about characterization of the class of ideals $\I$ such that for any $\I$-invariant injection $f\colon\omega\to\omega$ we have either $\fix(f)\in\I^\star$ or $f[\omega]\in\I$.

\begin{theorem}
The following are equivalent for any ideal $\I$ on $\omega$:
\begin{itemize}
	\item[(a)] there is an $\I$-invariant injection $f\colon\omega\to\omega$ with $\fix(f)\notin\I^\star$ and $f[\omega]\notin\I$;
	\item[(b)] there are $A,B\subseteq\omega$ such that $B\notin\I$, $A\triangle B\notin\I$ and $\I|A\sqsubseteq \I|B$;
	\item[(c)] $\I$ is not a maximal ideal. 
\end{itemize}
\end{theorem}

\begin{proof}
\textbf{(a) $\Rightarrow$ (c): }Let $\I$ be a maximal ideal. We will show that $\I$ does not satisfy condition (a). Suppose that there is an $\I$-invariant injection $f\colon\omega\to\omega$ with $\fix(f)\notin\I^\star$ and $f[\omega]\notin\I$. Then $C\in\I$ and $f^{-1}[C]\in\I^\star$ would contradict $f[\omega]\in\I^\star$, so actually $f$ is bi-$\I$-invariant. But then $\I$ cannot be maximal by Theorem \ref{bi-inv} and Example \ref{max}. A contradiction.

\textbf{(c) $\Rightarrow$ (b): }Take any not maximal ideal $\I$. Assume first that $\I$ is not dense. Take $C\subseteq\omega$ such that $\I|C\cong\fin$. Let $A$ and $B$ be two infinite and disjoint subsets of $C$. Then $A$ and $B$ witness that $\I$ satisfies condition (b).

Assume now that $\I$ is dense and take $A=\omega$ and any $B\notin\I\cup\I^\star$. Then $A\triangle B=B^c\notin\I$. Define $\phi\colon B\to\omega$ by $f(x)=x$ for all $x\in B$. Let us recall that if $\mathcal{J}_1$ is a dense ideal, then $\mathcal{J}_1\sqsubseteq\mathcal{J}_2$ if and only if there is a $1-1$ function $f\colon\bigcup\mathcal{J}_2\to\bigcup\mathcal{J}_1$ such that $f^{-1}[A]\in\mathcal{J}_2$ for all $A\in\mathcal{J}_1$ (cf. \cite{Katetov} or \cite{Farkas}). Therefore, by the above fact we get that $\I|A\sqsubseteq \I|B$. 

\textbf{(b) $\Rightarrow$ (a): }If $\I=\fin$, then obviously there is an $\I$-invariant injection $f\colon\omega\to\omega$ with $\fix(f)\notin\I^\star$ and $f[\omega]\notin\I$ (consider for instance the function given by $x\mapsto x+1$ for all $x\in\omega$). So we can suppose that $\I\neq\fin$. Take an infinite $C\in\I$. We can assume that $(A\cup B)\cap C=\emptyset$ (otherwise consider $A'=A\setminus C$ and $B'=B\setminus C$).

We are ready to define the $\I$-invariant injection $f$. Let $f\upharpoonright A$ be equal to the inverse of the function witnessing that $\I|A\sqsubseteq \I|B$ (so we already have $f[\omega]\supseteq f[A]=B\notin\I$), $f\upharpoonright (C\cup (B\setminus A))$ be any bijection between $C\cup (B\setminus A)$ and $C$, and $f\upharpoonright (\omega\setminus (A\cup B\cup C))$ be the identity function. Then $f$ is an $\I$-invariant injection. Moreover, $\fix(f)^c$ contains the set $A\triangle B\notin\I$, which does not belong to $\I$. Hence, $\fix(f)\notin\I^\star$. This finishes the entire proof.
\end{proof}

\section{Ideals induced by submeasures}
\label{submeasures}

In this section we answer \cite[Question 2]{Inv}. Firstly, we present its context.

A map $\phi:\mathcal{P}(\omega)\rightarrow[0,\infty]$ is a \emph{submeasure on $\omega$} if $\phi(\emptyset)=0$ and $\phi(A)\leq\phi(A\cup B)\leq\phi(A) + \phi(B)$, for all $A,B\subseteq \omega$. It is \emph{lower semicontinuous} if additionally $\phi(A) = \lim_{n\rightarrow\infty} \phi(A\cap \{0,\ldots,n\})$ for all $A\subseteq\omega$. For a lower semicontinuous submeasure $\phi$ on $\omega$ we define
$$\fin(\phi)=\{A\subseteq\omega: \phi(A)\textrm{ is finite}\},$$
$$\exh(\phi)=\left\{A\subseteq\omega:\ \lim_{n\rightarrow\infty}\phi(A\cap\{n,n+1,\ldots\})=0\right\}.$$

An ideal $\I$ is a \emph{P-ideal} if for every $(X_n)_{n\in \omega}\subseteq\mathcal{I}$ there is $X\in\mathcal{I}^\star$ with $X\cap X_n$ finite for all $n\in\omega$. 

For any lower semicontinuous submeasure $\phi$ on $\omega$ the ideal $\exh(\phi)$ is an $\mathcal{F}_{\sigma\delta}$ P-ideal and $\fin(\phi)$ is an $\mathcal{F}_{\sigma}$ ideal containing $\exh(\phi)$ \cite[Lemma 1.2.2]{Farah}.

\begin{proposition}[{\cite[Proposition 11]{Inv}}]
\label{FinExh}
Let $\varphi$ be a lower semi-continuous submeasure and $f\colon\omega\to\omega$ be an increasing injection such that there is $c_f>0$ with $\varphi(A)\geq c_f\varphi(f[A])$ for all $A\subseteq\omega$. Then $f$ is both $\fin(\varphi)$-invariant and $\exh(\varphi)$-invariant. Moreover, if there is $c'_f>0$ with $\varphi(A)\geq c'_f\varphi(f^{-1}[A])$ for all $A\subseteq\omega$, then $f$ is both bi-$\fin(\varphi)$-invariant and bi-$\exh(\varphi)$-invariant.
\end{proposition}

Set a function $f\colon\omega\to [0,\infty)$ such that $\sum_{n\in\omega}f(n)=+\infty$. For $A\subseteq\omega$ and an interval $I\subseteq\omega$ we denote $A_f(I)=\sum_{n\in I\cap A}f(n)$. Define the \emph{summable ideal associated with $f$} by the formula
$$\I_{(f(n))}=\left\{A\subseteq \omega:\ \sum_{n\in A}f(n)\textrm{ is finite}\right\}$$
and the \emph{Erd{\"o}s-Ulam ideal associated with $f$} by the formula
$$\eu_f=\left\{A\subseteq\omega:\ \limsup_{n\to\infty}\frac{A_f[1,n]}{\omega_f[1,n]}=0\right\}.$$
Then $\I_{(f(n))}=\exh(\phi)=\fin(\phi)$ for a lower semi-continuous submeasure $\phi:\mathcal{P}(\omega)\rightarrow[0,\infty]$ given by $\phi(A)=\sum_{n\in A}f(n)$ for all $A\subseteq\omega$. Hence, $\I_{(f(n))}$ is an $\mathcal{F}_{\sigma}$ P-ideal. What is more, $\eu_f=\exh(\varphi)$ for a lower semi-continuous submeasure $\varphi:\mathcal{P}(\omega)\rightarrow[0,\infty]$ given by 
$$\varphi(A)=\sup_{n\in\omega}\frac{A_f[1,n]}{\omega_f[1,n]}$$
for all $A\subseteq\omega$. Hence, $\eu_f$ is an $\mathcal{F}_{\sigma\delta}$ P-ideal.

\begin{proposition}[{\cite[Section 4]{Inv}}]
Let $g\colon\omega\to [0,+\infty)$ be such that $\sum_{n\in \omega}g(n)=+\infty$. Assume additionally that $g$ is nonincreasing. Then every increasing injection $f\colon\omega\to\omega$ is both $\eu_g$-invariant and $\I_{(g(n))}$-invariant. 
\end{proposition}

The next proposition shows that monotonicity condition imposed on $g$ is crucial.

\begin{proposition}[{\cite[Proposition 15]{Inv}}]
There are an Erd{\"o}s-Ulam ideal $\I$ and an increasing injection $f\colon\omega\to\omega$ such that neither $f$ nor $f^{-1}$ is $\I$-invariant.
\end{proposition}

We are ready to formulate our problem. Let $g\colon\omega \to [0,\infty)$ be nondecreasing. Is it true that the class of all increasing injections $f\colon\omega \to \omega$ which are bi-$\eu_g$-invariant equals the class of all increasing injections $f\colon\omega \to \omega$ which are bi-$\I_{(1/g(n))}$-invariant?

Note that originally in \cite[Question 2]{Inv} the function $g$ is increasing. However, in the following considerations it would be clearer to use nondecreasing function instead of the increasing one. We can always modify nondecreasing $g$ to increasing $g'\geq g$ without changing the original ideal by making sure that $\sum_{n\in\omega} g'(n)-g(n)$ is convergent.

The next theorem shows that the answer is positive in the case of $g(n)=n$.

\begin{theorem}[{\cite[Corollary 20]{Inv}}]
\label{IdI1n}
Let $f\colon\omega\to\omega$ be an increasing injection. The following are equivalent:
\begin{itemize}
	\item $f$ is bi-$\I_d$-invariant;
	\item $\underline{d}(f[\omega])>0$;
	\item there is $c\in\omega$ such that $f(n)\leq Cn$ for all $n\geq 1$;
	\item $f$ is bi-$\I_{(1/n)}$-invariant.
\end{itemize}
\end{theorem}

We answer our problem in negative.

\begin{proposition}
There is a nondecreasing $g\colon\omega \to [0,\infty)$ such that the classes of all increasing injections $f\colon\omega \to \omega$ which are bi-$\eu_g$-invariant and of all increasing injections $f\colon\omega \to \omega$ which are bi-$\I_{(1/g(n))}$-invariant are not equal.
\end{proposition}

\begin{proof}
If $g : \omega \to [0,\infty)$ is a nondecreasing bounded function, then $\eu_g=\I_d$, while $\I_{(1/g(n))}=\fin$. Since there are increasing functions that are not bi-$\I_d$-invariant, the classes of their increasing bi-$\I$-invariant functions are different.
\end{proof}

We can weaken the above property in the following way. We say that an Erd{\"o}s-Ulam ideal $\I$ satisfies condition $(C2)$, if there is such nondecreasing function $g : \omega \to [0,\infty)$ that $\I=\eu_g $ and the class of all increasing functions $f : \omega \to \omega$ which are bi-$\eu_g$-invariant equals the class of all increasing injections $f : \omega \to \omega$ which are bi-$\I_{(1/g(n))}$-invariant. Is it true that all Erd{\"o}s-Ulam ideals $\I$ satisfy condition $(C2)$?

We have two counterexamples. The first one, however, is not dense.

\begin{proposition}
There is an Erd{\"o}s-Ulam ideal which does not satisfy condition $(C2)$.
\end{proposition}

\begin{proof}
Let $(I_n)_{n\in\omega}$ be a sequence of consecutive intervals such that $|I_n|=(2^n)!$ for each $n$. Denote $\I=\eu_h$, where $h(k)=(2^n)!$ for all $k\in I_n$. 

Since all summable ideals defined by monotonic functions are transitive, the increasing injection $f\colon\omega\to\omega$ defined by $f(n)=n+1$, for all $n$, is bi-$\I_{(1/g(n))}$-invariant for every nondecreasing function $g$. However, we will show that this function is not bi-$\I$-invariant.

Let $A=\{\max I_n:\ n\in\omega\}$. Then $A\in\I$, since for each $k\in I_n\cap A$ we have 
$$\frac{A_h([1,k])}{\omega_h([1,k])}\leq \frac{n\cdot (2^n)!}{((2^n)!)^2}\xrightarrow{n\rightarrow\infty} 0.$$
On the other hand, for $k\in I_{n+1}\cap f[A]$ we have 
$$\frac{f[A]_h[1,k]}{\omega_h[1,k]}\geq \frac{(2^{n+1})!}{((2^{n+1})!+n\cdot ((2^n)!)^2)}.$$ 
Since $\frac{(2^{n+1})!}{((2^n)!)^2}\geq 2^{2^n} $, the right hand side of the above inequality tends to $1$. Thus, $f[A]\not\in \I$ and $f$ is not bi-$\I$-invariant.
\end{proof}

\begin{theorem}
There is a dense Erd{\"o}s-Ulam ideal which does not satisfy condition $(C2)$.
\end{theorem}

\begin{proof}
Let $(k_n)_{n\in\omega}$ be a sequence defined by $k_0=0$ and 
$$k_n=\min\left\{x\in\omega:\ \frac{2^{x}}{2^{k_{n-1}}}\geq n\right\}$$
for all $n>0$. Define $I_n=[2^{k_n},2^{k_{n+1}})$ and $\I=\eu_h$, where $h\colon\omega\to\omega$ is given by $h(i)=2^{k_n}$ for all $i\in I_n$. 

Notice that for $i\in I_n$ we have
$$\frac{h(i)}{\omega_h[1,i]}\leq \frac{2^{k_n}}{(2^{k_{n-1}})^2}\leq \frac{2n}{2^{k_{n-1}}}\xrightarrow{n\rightarrow\infty}0.$$
Therefore, $\I$ is dense.

For $x\in\mathbb{R}$ by $[x]$ we denote the nearest integer to $x$. Observe that given any sequence of positive reals $(a_n)_{n\in\omega}$ bounded by $1$, if $A$ denotes the set consisting of the first $[a_n 2^{k_n}]$ elements of each $I_{n+1}$, then 
\begin{equation}
  A\in\I \Longleftrightarrow \lim_{n\to\infty} a_n=0,	\label{eq:0}
\end{equation}
since 
$$\frac{A_h(I_{n+1})}{\omega_h(I_n)}=\frac{ a_n \cdot 2^{k_n}\cdot 2^{k_{n+1}}}{2^{k_n}\cdot (2^{k_{n+1}}-2^{k_n}) }\approx a_n \cdot \frac{n+1}{n},  $$
while $\frac{A_h(I_n)}{\omega_h(I_n)}\leq \frac{a_n}{n}$.

Let $g : \omega \to [0,\infty)$ be a nondecreasing function such that $\I=\eu_g $ and let $\J=\I_{(1/g(n))}$ be the appropriate summable ideal. Consider sets $B_b$ for $b\in(0,1]$ consisting of the last $[b 2^{k_n}]$ elements of each interval $I_n$. We have two cases:
\begin{itemize}
\item[1.]  there is $b\in(0,1]$ such that $B_b\in\J$,
\item[2.]  $B_b\not\in\J$ for all $b\in (0,1]$.
\end{itemize}

{\bf Case 1.: }Consider the set $C=B_b\in \J$ and define $f\colon\omega\to\omega$ by 
$$f(x)=\left\{\begin{array}{ll}
0, & \textrm{if } x=0,\\
2^{k_{n+1}}, & \textrm{if } x=\min(I_n\cap C)\textrm{ for some }n,\\
f(x-1)+1, & \textrm{otherwise}.
\end{array}\right.$$

Observe that $f[C]$ consists of the first $[b 2^{k_n}]$ elements of each $I_{n+1}$. Hence, $f[C]\not\in\I$ by (\ref{eq:0}). On the other hand, $\frac{C_h(I_n)}{\omega_h(I_n)}\leq \frac{b}{n}$, so $C\in\I$. Thus, $f$ is not bi-$\I$-invariant.

Since the sequence $\left(\frac{1}{g(n)}\right)_{n\in\omega}$ is nonincreasing and $f$ is increasing, $f$ is obviously $\J$-invariant. We will show that $f^{-1}$ is $\J$-invariant as well. Set $D\not\in \J$. We only need to see that 
$$\sum_{i\in f[D]} \frac{1}{g(i)}\geq  \sum_{i\in D} \frac{1}{g(i)} - \sum_{i\in C} \frac{1}{g(i)},$$
since $\left(\frac{1}{g(n)}\right)_{n\in\omega}$ is nonincreasing and $\sum_{i\in C} \frac{1}{g(i)}$ is convergent. Therefore, $f[D]\not\in\J$ whenever $D\not\in \J$.

{\bf Case 2.: }
First, observe that for all $b\in(0,1]$ the ratio 
$$\frac{(B_b+b2^{k_n})_g(I_{n+1})}{(B_b)_g(I_n)}$$ 
has to tend to infinity in order to maintain $\I=\eu_g$. Indeed, if it would be bounded on some subsequence $(i_n)_{n\in\omega}$, then each set formed of some first elements of each $I_{i_n}$ would be in $\eu_g$ if and only if appropriate subset of $B_b$ would be in $\eu_g$. That is not the case for $\I$, since $B_b\in \I$ for all $b\leq 1$ while $\bigcup_{n\in\omega}((B_b\cap I_n)+b2^{k_n})\not\in \I $ by (\ref{eq:0}) for all $b\leq 1$. Therefore, for all $M>0$ and $c\in (0,1]$, there is such $N \in \omega$ that for all $n\geq N$ we have $\frac{g(\max I_n + c2^{k_n})}{g(\max I_n)}\geq M$.

Next we find such a nonincreasing sequence $(b_n)_{n\in\omega}$ tending to $0$ that the set $B$ consisting of the last $[b_n 2^{k_n}]$ elements of each interval $I_n$ does not belong to $\J$ while the sequence $(M_n)_{n\in\omega}$, where $M_n=\frac{g(\max I_n + b_n 2^{k_n})}{g(\max I_n)}$ for each $n\in\omega$, tends to infinity.

Due to the Abel-Dini Theorem \cite[Theorem 173]{Knopp}, which says that when $\sum_{n\in\omega}x_n$ diverges, then $\sum_{n\in\omega}\frac{x_n}{(x_1+\ldots +x_n)^{1+\delta}}$ converges for all $\delta>0$, we can find a sequence $(c_n)_{n\in\omega}$ such that $\sum_{n\in\omega}c_n B_{1/g}(I_n)$ diverges while $\sum_{n\in\omega}\frac{c_n B_{1/g}(I_n)}{M_n}$ converges. To do that, we only need to make sure that $(c_1 B_{1/g}(I_1) +\ldots +c_n B_{1/g}(I_n))^2$ tends to infinity and is not greater than $M_n$ at the same time.

Define the set $C$ as a union of the first $[c_n |B\cap I_n|]$ elements of each $B\cap I_n$. Let $f\colon\omega\to\omega$ be given by 
$$f(x)=\left\{\begin{array}{ll}
0, & \textrm{if } x=0,\\
2^{k_{n+1}}+[b_n 2^{k_n} ]+[(1-c_n)|B\cap I_n|], & \textrm{if } x=\min(I_n\cap C)\textrm{ for some }n,\\
f(x-1)+1, & \textrm{otherwise}.
\end{array}\right.$$

Observe that since $\sum_{n\in\omega} C_{1/g}(I_n)\geq \sum_{n\in\omega}c_n B_{1/g}(I_n)$, the set $C$ does not belong to $\J$. Moreover, we can see that $\sum_{n\in\omega} f(C)_{1/g}(I_{n+1})\leq \sum_{n\in\omega}\frac{c_n B_{1/g}(I_n)}{M_n}$, thus $f[C]\in \J$. Hence, $f$ is not bi-$\J$-invariant. 

To show that $f$ is bi-$\I$-invariant, pick $D\subseteq \omega$. Both $B$ and $f[B]$ belong to $\I$ (by (\ref{eq:0})), so without loss of generality we may assume that $D\cap B=\emptyset$. Firstly, consider the case that $D\in \I$. It is easy to see that $f[D]\in\I$, since for $i\in D\cap I_n$ we have $f(i)\in I_n$ and the function $h$ is constant on $I_n$. On the other hand, if $D\not \in \I$ then there are $\alpha>0$ and infinitely many $i\in\omega$ such that $\frac{D_h[1,i]}{\omega_h[1,i]}>\alpha $. Then for such $i\in I_n$ we have
$$\frac{f[D]_h[1,i]}{\omega_h[1,i]}\geq  \frac{D_h[1,i-2b_n 2^{k_{n-1}}]}{\omega_h[1,i]}> \alpha - \frac{2b_n \cdot 2^{k_{n-1}}\cdot 2^{k_{n}} }{2^{k_{n-1}}\cdot (2^{k_{n}} - 2^{k_{n-1}} )}.$$
To finish the proof it remains to notice that the right hand side of the above inequality is greater than $\alpha/2$ for almost all $n$. Therefore, $f[D]\not\in\I$.
\end{proof}

We end this section with an example of a "nice" anti-homogeneous ideal. All examples of anti-homogeneous ideals presented in Section \ref{maximal} were not "nice" (i.e., they were not Borel or even analytic). The ideal presented below is an Erd{\"o}s-Ulam anti-homogeneous ideal. 

\begin{theorem}
\label{ex}
There is an Erd{\"o}s-Ulam anti-homogeneous ideal.
\end{theorem}

\begin{proof}
Let $(I_n)_{n\in\omega}$ be a family of consecutive intervals such that each $I_n$ has length $n!$. Let also $(\varphi_n)_{n\in\omega}$ be a family of measures on $\omega$ given by:
$$\varphi_n(\{k\})=\left\{\begin{array}{ll}
\frac{1}{n!}, & \textrm{if } k\in I_n\\
0, & \textrm{if } k\in\omega\setminus I_n.
\end{array}\right.$$
Consider the ideal $\I=\{A\subset\omega:\ \lim_{n\to\infty}\varphi_n(A)=0\}$. This is a density ideal (in the sense of Farah, cf. \cite[Chapter 1.13]{Farah}). Moreover, it is an Erd{\"o}s-Ulam ideal by \cite[Theorem 1.13.3]{Farah}. We will show that $\I$ is anti-homogeneous.

Take any $B\notin\I\cup\I^\star$. Then there are $M>0$ and some infinite $T\subset\omega\setminus M$ such that $\varphi_n(B)\geq\frac{1}{M}$ for all $n\in T$. For each $n\in T$ pick $A_n\subset I_n\cap B$ such that $|A_n|=\frac{n!}{M}$. We will show that $A=\omega\setminus\bigcup_{n\in T}A_n$ is not in $H(\I)$. By Theorem \ref{supersets} it will follow that $B^c\notin H(\I)$.

Assume that $A\in H(\I)$ and $f\colon\omega\to A$ is a bijection witnessing it. Denote: 
\begin{eqnarray}
\omega^+ & = & \bigcup_{n\in\omega}\{x\in I_n:\ f(x)>\max I_n\}, \nonumber \\
\omega^- & = & \bigcup_{n\in\omega}\{x\in I_n:\ f(x)<\min I_n\}, \nonumber \\
\omega^= & = & \bigcup_{n\in\omega}\{x\in I_n:\ f(x)\in I_n\}. \nonumber 
\end{eqnarray}
Observe that $\varphi_n(\omega^-)=\varphi_n(\omega^-\cap I_n)<\frac{2}{n}$ for each $n$. Hence, $\omega^-\in\I$. Note also that $\omega^=\subseteq A\notin\I^\star$. Therefore, $\omega^+\notin\I$.

On the other hand, 
$$\varphi_n\left(f\left[\omega^+\right]\right)=\varphi_n\left(f\left[\bigcup_{k<n}I_k\cap\omega^+\right]\right)<\frac{2}{n}$$
for each $n$. Hence, $f[\omega^+]\in\I$ which contradicts the fact that $f$ witnesses $A\in H(\I)$. Therefore, $A\notin H(\I)$.
\end{proof}

\section{Ideal convergence}

In this section we consider ideal convergence and answer \cite[Question 3]{Inv} and \cite[Question 4]{Inv}.

Let $\I$ be an ideal on $\omega$. Recall that a sequence of reals $(x_n)_{n\in \omega}$ is \emph{$\I$-convergent} to $x\in \R$, if $\{n\in \omega:\ |x_n-x|\geq \varepsilon\}\in \I$ for any $\varepsilon>0$. We say that an ideal $\I$ on $\omega$ satisfies condition $(C3)$, if for any sequence $(x_n)_{n\in\omega}$ of reals $\I$-convergence of $(x_n)_{n\in\omega}$ to some $x\in\mathbb{R}$ implies convergence of $ (x_{f(n)})_{n\in\omega} $ to $x$ (in the classical sense) for some bi-$\I$-invariant injection $f$. \cite[Question 3]{Inv} asks about characterization of ideals satisfying condition $(C3)$.

It is easy to see that $\fin \oplus \cP(\omega)$ does not satisfy condition $(C3)$.

Recall that an ideal $\I$ is a \emph{P-ideal} if for every $(X_n)_{n\in \omega}\subseteq\mathcal{I}$ there is $X\in\mathcal{I}^\star$ with $X\cap X_n$ finite for all $n\in\omega$. Similarly, $\mathcal{I}$ is a \emph{weak P-ideal} if for every $(X_n)_{n\in \omega}\subseteq\mathcal{I}$ there is $X\notin\mathcal{I}$ with $X\cap X_n$ finite for all $n\in\omega$.

\begin{proposition}[{\cite[Proposition 22]{Inv}}]
All admissible P-ideals satisfy condition $(C3)$.
\end{proposition}

\begin{proposition}[{\cite[Proposition 23]{Inv}}]
All ideals that are not weak P-ideals do not satisfy condition $(C3)$.
\end{proposition}

Now we give a characterization of ideals satisfying condition $(C3)$.

\begin{proposition}
The following are equivalent for any admissible ideal $\I$ on $\omega$:
\begin{itemize}
	\item[(a)] $\I$ satisfies condition $(C3)$;
	\item[(b)] for every countable family $\{A_n:\ n\in\omega\}\subset\I$ there exists such $A\in H(\I)$ that $A\cap A_n$ is finite for every $n\in\omega$.
\end{itemize}
\end{proposition}

\begin{proof}
\textbf{(a) $\Rightarrow$ (b): }Pick a countable family $\{A_n:\ n\in\omega\}\subset\I$ and suppose that for every $A\in H(\I)$ there is some $n\in\omega$ with $A\cap A_n$ infinite. Consider the sequence $(x_n)_{n\in\omega}$ defined by 
$$x_n=\left\{\begin{array}{ll}
\frac{1}{k}, & \textrm{if } n\in A_k\setminus \bigcup_{m<k}A_m \textrm{ for some }k,\\
0, & \textrm{otherwise}.
\end{array}\right.$$
Then the sequence $(x_n)_{n\in\omega}$ is clearly $\I$-convergent to $0$. However, for every bi-$\I$-invariant injection $f$ there is such $k\in\omega$ that $f[\omega]\cap A_k$ is infinite (since $f[\omega]\in H(\I)$). Thus, there are infinitely many elements $n\in\omega$ for which $x_{f(n)}\geq 1/k$. Therefore $(x_{f(n)})_{n\in\omega}$ is not convergent to $0$.

\textbf{(b) $\Rightarrow$ (a): }Let $(x_n)_{n\in\omega}$ be a sequence of reals $\I$-convergent to some $x\in\mathbb{R}$. Define $A_k=\{n\in\omega:\ |x_n-x|>\frac{1}{k}\}$. We can find such bi-$\I$-invariant injection $f$ that $f[A] \cap A_k $ is finite for every $k$. Then $(x_{f(n)})_{n\in\omega}$ is convergent to $x$.
\end{proof}

\begin{remark}
An admissible homogeneous ideal satisfies condition $(C3)$ if and only if it is a weak P-ideal. Moreover, an anti-homogeneous ideal satisfies condition $(C3)$ if and only if it is a P-ideal.
\end{remark}

Now we move to another problem. It is known that for a sequence of reals $(x_n)_{n\in\omega}$ and some $x\in\mathbb{R}$, if any sequence of indices $(n_k)_{k\in\omega}$ contains a subsequence $(n_{k_l})_{l\in\omega}$ such that $(x_{n_{k_l}})_{l\in\omega}$ converges to $x$, then the whole sequence $(x_n)_{n\in\omega}$ converges to $x$ as well. We are interested in ideals for which ideal version of the above fact holds. Namely, say that an ideal $\I$ on $\omega$ satisfies condition $(C4)$ if for any sequence $(x_n)_{n\in\omega}$ of reals and $x\in\mathbb{R}$ the fact that for every bi-$\I$-invariant $f\colon\omega \to \omega$ there is a bi-$\I$-invariant $g\colon\omega \to \omega$ such that $(x_{g(f(n))})$ is $\I$-convergent to $x$ implies that $(x_n)_{n\in\omega}$ is $\I$-convergent to $x$. \cite[Question 4]{Inv} concerns characterization of ideals satisfying condition $(C4)$. In \cite{Inv} it is pointed out that condition $(C1)$ implies condition $(C4)$.

\begin{proposition}
The following are equivalent for any ideal $\I$ on $\omega$:
\begin{itemize}
	\item[(a)] $\I$ satisfies condition $(C4)$; 
	\item[(b)] $\I$ is anti-homogeneous.
\end{itemize}
\end{proposition}

\begin{proof}
\textbf{(a) $\Rightarrow$ (b): }Suppose that $\I$ is not anti-homogeneous, i.e., there is $A\in H(\I)\setminus\I^\star$. Let $g\colon\omega \to A$ be the isomorphism witnessing that $\I|A\cong\I$. Then $g$ is bi-$\I$-invariant. Define 
$$x_n=\left\{\begin{array}{ll}
1, & \textrm{if } n\in A,\\
0, & \textrm{if } n\in \omega\setminus A.
\end{array}\right.$$
Then $(x_n)_{n\in\omega}$ is clearly not $\I$-convergent. On the other hand, for every bi-$\I$-invariant $f\colon\omega \to \omega$ we have $g[f[\omega]]\subseteq A$. Thus, $x_{g(f(n))}$ is constant, and hence convergent, sequence. A contradiction with the assumption. Therefore, $\I$ is anti-homogeneous.

\textbf{(b) $\Rightarrow$ (a): }Since $\I$ is anti-homogeneous, for every bi-$\I$-invariant functions $f,g\colon\omega \to \omega$ we have $g[f[\omega]]\in\I^\star$. To finish the proof it is sufficient to observe that if $(x_n)_{n\in A}$ is $\I$-convergent to $x$ for some $A\in \I^\star$, then $(x_n)_{n\in\omega}$ is obviously $\I$-convergent to $x$ as well.
\end{proof}

Notice that $\I_d$ does not satisfy condition $(C4)$ (since $\I_d|\{0,2,4,\ldots\}\approx \I_d$). However, we can change the domain and codomain of $g$ and investigate a slightly weaker property. We say that an ideal $\I$ on $\omega$ satisfies condition $(C5)$, if for any sequence $(x_n)_{n\in\omega}$ of reals and $x\in\mathbb{R}$ the fact that for every bi-$\I$-invariant $f\colon\omega \to \omega$ there is a bi-$\I$-invariant $g\colon f[\omega] \to f[\omega]$ such that $(x_{g(f(n))})$ is $\I$-convergent to $x$ implies that $(x_n)_{n\in\omega}$ is $\I$-convergent to $x$.

\begin{remark}
Ideal $\fin \oplus \cP(\omega)$ satisfies condition $(C5)$. However, it does not satisfy condition $(C4)$.
\end{remark}

\begin{proposition}
The following are equivalent for any admissible ideal $\I$ on $\omega$:
\begin{itemize}
	\item[(a)] $\I$ satisfies condition $(C5)$;
	\item[(b)] for every $A\not\in\I\cup \I^\star$ there is $B\in H(\I)$ such that for all $C\in H(\I)$ with $C\subseteq B$ we have $A\cap C \not\in \I$.
\end{itemize}
\end{proposition}

\begin{proof}
\textbf{(a) $\Rightarrow$ (b): }Suppose that there is $A\not\in\I\cup \I^\star$ such that for each $B\in H(\I)$ there is $C\in H(\I)$ with $C\subseteq B$ and $A\cap C\in \I$. Define 
$$x_n=\left\{\begin{array}{ll}
1, & \textrm{if } n\in A,\\
0, & \textrm{if } n\in \omega\setminus A.
\end{array}\right.$$
Then clearly $(x_n)_{n\in\omega}$ is not $\I$-convergent. On the other hand, let $f\colon\omega \to \omega$ be any bi-$\I$-invariant function and denote $B=f[\omega]$. By our assumption there is $C\in H(\I)$ with $C\subseteq B$ and $A\cap C\in \I$. Let $g\colon B \to B$ be a bi-$\I$-invariant function such that $g[B]=C$. Observe that $(x_{g(f(n))})_{n\in\omega}$ is $\I$-convergent to $0$, since $A\cap C\in \I$. A contradiction with the assumption that $\I$ satisfies condition $(C5)$.

\textbf{(b) $\Rightarrow$ (a): }Suppose that $\I$ does not satisfy condition $(C5)$ and let $(x_n)_{n\in\omega}$ be such a sequence of reals that for any bi-$\I$-invariant $f\colon\omega \to \omega$
there is a bi-$\I$-invariant $g\colon f[\omega] \to f[\omega]$ such that $(x_{g(f(n))})_{n\in\omega}$ is $\I$-convergent to $x$, but it is not $\I$-convergent to $x$. Then there is an $\epsilon>0$ such that $A=\{n\in\omega:\ |x_n-x|>\epsilon\}\not\in\I\cup\I^\star$. 

By our assumption there is $B\in H(\I)$ such that for all $C\in H(\I)$ with $C\subseteq B$ we have $A\cap C \not\in \I$. Let $f\colon\omega \to \omega$ be a bi-$\I$-invariant function witnessing that $B\in H(\I)$, i.e., $f[\omega]=B$. Then for every bi-$\I$-invariant $g\colon B \to B$ we have $g[B] \cap A\not\in\I$. Thus, $\{n\in g[B]:\ |x_n-x| >\epsilon  \}  \not\in\I$. Therefore, $(x_{g(f(n))})_{n\in\omega}$ is not $\I$-convergent to $x$. A contradiction. Hence, $\I$ satisfies condition $(C5)$.
\end{proof}

\begin{remark}
Note that all homogeneous and anti-homogeneous ideals satisfy condition $(C5)$.
\end{remark}

We will show that the ideal $\I_d$ does not satisfy condition $(C5)$. We need the following fact.

\begin{theorem}
\label{Id}
Let $A\in H(\I_d)$ and $\{a_0,a_1\ldots\}$ be an increasing enumeration of $A$. Then the function $f\colon\omega\to A$ given by $f_A(n)=a_n$ witnesses that $\I_d|A\cong\I_d$.
\end{theorem}

\begin{proof}
Suppose that $f(n)=a_n$ is not an isomorphism between $\I_d|A$ and $\I_d$. Let $g\colon\omega \to A$ be such isomorphism. Since every increasing function is $\I_d$-invariant, there is $B\notin\I_d$ such that $f[B]\in\I_d$. Let $\{b_0,b_1,\ldots \}$ be an increasing enumeration of $B$.

We define inductively sets $B_n$ for $n\in\omega$. Let $B_0$ consist elements $b_i^0$ such that 
$b_0^0=\min\{x\in \omega\setminus B :\ x>b_{0}\}$ and
$$b_{i}^0=\min\{x\in \omega\setminus B :\ x>b_{i}\ \wedge\ x>b_{0}^{i-1} \}$$
for all $i>0$. Then $f(B_0)\in\I_d$ as well. Suppose now that $B_0,B_1,\ldots,B_{n-1}$ are already constructed. Let $b_0^n=\min\{x\in \omega\setminus (B \cup B_0 \cup\ldots \cup B_{n-1}) :\ x>b_0^{n-1}\}$ and
$$b_{i}^n=\min\{x\in \omega\setminus (B \cup B_0 \cup\ldots \cup B_{n-1}) :\ x>b_i^{n-1}\ \wedge\ x>b_{i-1}^n \}$$
for all $i>0$. Let $B_n$ consist of all $b_i^n$ for $i\in\omega$. It is easy to see that $f[B_n]\in\I_d$.

Define intervals $I_n=\{2^{2n},2^{2n}+1,\ldots,2^{2n+2}-1\}$ for all $n\in\omega$. Since $B\notin\I_d$, there are $M>0$ and infinitely many such $n\in\omega$ that $\frac{|B \cap I_n|}{|I_n|}>\frac{1}{M}$. Therefore, there is some $k\in\omega$ such that $I_n\subseteq B\cup B_0\cup\ldots\cup B_k$ for infinitely many $n\in\omega$ (not necessarily the same as above). Let $C$ be the union of those $I_n$'s and define:
\begin{eqnarray}
C^= & = & \bigcup_{n\in\omega}\{ c\in C:\ c \in I_n\ \wedge\ f(\min I_n)\leq g(c)\leq f(\max I_n)\}, \nonumber \\
C^+ & = & \bigcup_{n\in\omega}\{ c\in C:\ c \in I_n\ \wedge\ g(c) > f(\max I_n)\}, \nonumber \\
C^- & = & \bigcup_{n\in\omega}\{ c\in C:\ c \in I_n\ \wedge\ g(c) < f(\min I_n)\}. \nonumber 
\end{eqnarray}

Observe that $f[C]\in\I_d$. Hence, $g[C^=]$ belongs to $\I_d$. Moreover, $|g[C^+]\cap n| < |f[C]\cap n| $, so $g[C^+]$ belongs to $\I_d$ as well. Thus, $C^+\cup C^=\in \I_d$.

On the other hand, if $C\supseteq I_n $ for some $n$ then $\frac{|(C\setminus C^-) \cap I_n|}{|I_n|}\geq 2/3 $, since $\frac{|C^-\cap I_n|}{|I_n|}\leq 1/3$. Thus, $C\setminus C^-\not\in \I_d$. A contradiction.
\end{proof}

\begin{corollary}
\label{Idwn}
For any $B\subset\omega$ we have $B\in H(\I_d)$ if and only if $\underline{d}(B)>0$.
\end{corollary}

\begin{proof}
It follows from Theorems \ref{Id} and \ref{IdI1n}.
\end{proof}

\begin{problem}
Characterize ideals $\I$ such that for any $A\in H(\I)$ the function $f_A\colon\omega\to A$ given by $f_A(n)=a_n$, where $\{a_0,a_1\ldots\}$ is an increasing enumeration of $A$, witnesses that $\I|A\cong\I$.
\end{problem}

\begin{proposition}
Ideal $\I_d$ does not satisfy condition $(C5)$.
\end{proposition}

\begin{proof}
Pick $A\not\in\I_d$ such that $\underline{d}(A)=0$. For every $B\in H(\I_d)$ we have $\underline{d}(B\setminus A)>0$, since $\underline{d}(B)>0$ (by Corollary \ref{Idwn}) and $\underline{d}(B)\leq \underline{d}(A)+\underline{d}(B\setminus A)$. Thus, there is a set $C=B\setminus A\subseteq B$ which is in $H(\I_d)$ (by Corollary \ref{Idwn}) while $C\cap A=\emptyset\in\I_d$.
\end{proof}

\end{document}